\documentclass[12 pt]{amsart}
\usepackage{setspace, amsmath, amsthm, amssymb, amsfonts, amscd, epic, graphicx, ulem, dsfont}
\usepackage[T1]{fontenc}
\usepackage{multirow}
\usepackage{bbm}

\newtheorem{thm}{Theorem}
\newtheorem*{un-thm}{Theorem}
\newtheorem*{con}{Conjeture}

\newtheorem{cor}{Corollary}

\theoremstyle{abstract}

\theoremstyle{remark}

\theoremstyle{definition}

\newcommand{\re}{\text{\rm{Re}}}

\makeatletter \@namedef{subjclassname@2010}{
  \textup{2010} Mathematics Subject Classification}
\makeatother

\begin{document}

\title[The Conjecture of Fong-Tsui]{A Contribution to the Fong-Tsui Conjecture Related to Self-adjoint Operators}
\author[M. H. MORTAD]{MOHAMMED HICHEM MORTAD}

\dedicatory{}
\date{}
\thanks{}


\address{
D\'{e}partement de Math\'{e}matiques, Universit\'{e} d'Oran
(Es-Senia), B.P. 1524, El Menouar, Oran 31000, Algeria.\newline {\bf
Mailing address}: \newline Dr Mohammed Hichem Mortad \newline BP
7085 Es-Seddikia\newline Oran \newline 31013 \newline Algeria}

\email{mhmortad@gmail.com, mortad@univ-oran.dz.}

\begin{abstract}
We are interested in an open question raised by Fong-Tsui (dating
back to the beginning of the eighties of last century) as to whether
a bounded operator whose absolute value is less than the absolute
value of its real part is self-adjoint. The analogue in the
unbounded operators setting is also treated.
\end{abstract}

\maketitle

\section{Notations and Terminology}

If $T$ is a linear operator on a complex Hilbert space, then $\re T$
denotes the real part of $T$, that is, the operator
$\frac{T+T^*}{2}$, where $T^*$ is the adjoint of $T$.

The absolute value of $T$, denoted by $|T|$, is the positive square
root of the positive operator $T^*T$. Normal, self-adjoint, positive
operators and isometries are defined in their usual fashion.

If $T$ is an operator such that $T+T^*$ commutes with $T^*T$, then
we say that $T$ belongs to the $\Theta$-class (a class of operators
introduced by S.L. Campbell, see \cite{Campbell 1975 T+t* comm
T*T}).

For basic results on operator theory, the reader may consult
\cite{Con} or \cite{RUD}.

\section{Essential Background}
In the present section we recall the main results which will be
needed to prove our theorems.

\begin{thm}[\cite{Fong-Istratescu 1979}]\label{Fong-Istra}
If $T$ is a bounded operator verifying $|T|^2\leq (\re T)^2$, then
$T$ is automatically self-adjoint.
\end{thm}

\begin{thm}[\cite{Fong-Tsui 1981}]\label{Fong Tsui}
If $T$ is a bounded operator satisfying $|T|\leq \re T$, then $T$ is
positive.
\end{thm}

\section{Positive Results}
The following conjecture appeared in \cite{Fong-Tsui 1981} (which
was a try to strengthen a result which appeared in
\cite{Fong-Istratescu 1979})

\begin{con}[Fong-Tsui \cite{Fong-Tsui 1981}]
If $T$ is a bounded operator and $|T|\leq |\re T|$, then $T$ is
self-adjoint.
\end{con}

The authors in \cite{Fong-Tsui 1981} gave some cases where this
result holds, eg if $T$ is compact or if the Hilbert space is
finite-dimensional (among others).

In the present paper we provide a modest contribution to the
conjecture in order to maintain the hope towards a complete solution
of it. We first give some classes for which the conjecture is true,
then, and in the end, we give an unbounded counterexample that shows
that the conjecture does not hold for unbounded operators.

Here is the first result:

\begin{thm}
Let $T$ be a linear operator on a Hilbert space. Let $U$ be the
unitary operator intervening in the polar decomposition of the
self-adjoint operator $\re T$. If $UT=TU^*$, then $|T|\leq |\re T|$
implies that $T$ is self-adjoint.
\end{thm}

\begin{proof}Let $U$ be a unitary operator such that $\re T=U|\re T|$. We may then write
\[|T|\leq |\re T|=U^*\re T.\]
From the hypothesis $UT=TU^*$ and by the unitarity of $U$ (or by the
Fuglede theorem), we get $U^*T=TU$. Moreover, we obtain
\[U^*\re T=\re(U^*T).\]
Observe that
\[|U^*T|=\sqrt{(U^*T)^*(U^*T)}=\sqrt{T^*UU^*T}=\sqrt{T^*T}=|T|.\]
So the hypothesis $|T|\leq |\re T|$ now looks like
\[|U^*T|\leq \re (U^*T).\]
Theorem \ref{Fong Tsui} yields the positiveness of $U^*T$ or simply
its self-adjointness. Calling on again the hypothesis $UT=TU^*$, we
immediately see that
\[T^*U=(U^*T)^*=T^*U=U^*T=TU,\]
which implies that $T$ is self-adjoint thanks to the unitarity of
$U$. This finishes the proof.
\end{proof}

The condition $UT=TU^*$ in the previous theorem can be dropped at
the cost of assuming that $T$ is in the $\Theta$-class. We have:

\begin{thm}
Let $T$ be a bounded operator belonging to the $\Theta$-class. Then
$T$ is self-adjoint whenever $|T|\leq |\re T|$.
\end{thm}

\begin{proof}
Since $T+T^*$ and $T^*T$ commute, $\re T$ commutes with $|T|$. Hence
the self-adjointness of $\re T$ implies that $|T|$ commutes with
$|\re T|$, i.e.
\[|T||\re T|=|\re T||T|.\]
Since $|T|\leq |\re T|$ (and $|T|\geq 0$!), the previous displayed
equation implies that
\[|T|^2\leq |T||\re T|\text{ and } |T||\re T|\leq |\re T|^2=(\re T)^2\]
or just
\[|T|^2\leq (\re T)^2.\]
The self-adjointness of $T$ now follows from Theorem
\ref{Fong-Istra}, completing the proof.
\end{proof}

It is obvious that normal operators $T$ are in the $\Theta$-class.
Therefore we have the

\begin{cor}\label{|T| leq |re T| T NORMAL}
Let $T$ be a normal operator such that $|T|\leq |\re T|$. Then $T$
is self-adjoint.
\end{cor}

Similarly, isometries do lie in the $\Theta$-class and so we have
the

\begin{cor}\label{|T| leq |re T| T ISOMETRY}
Let $T$ be an isometry operator verifying $|T|\leq |\re T|$. Then
$T$ is self-adjoint.
\end{cor}

\section{The Unbounded Analogue}

In this section we present an \textit{unbounded} counterexample to
the conjecture. This was probably expected due to the hazardous
terrain of domains of unbounded operators. Here is the
counterexample:

Let $S$ be an unbounded operator $S$, defined in a Hilbert space
$\mathcal{H}$ say, with domain $D(S)\subsetneq \mathcal{H}$. Set
$T=S-S$. Then $T=0$ \textit{on} $D(T)=D(S)\subsetneq \mathcal{H}$.

It is clear that $T$ \textit{cannot be self-adjoint} as it is not
closed (it is, however, symmetric) and that $D(T^*)=\mathcal{H}$.
Nonetheless, we have the following
\[T^*T(f)=0 \text{ on } D(T^*T)=\{f\in D(T):~0\in D(T^*)\}=D(T)\]
and
\[\left(\frac{T+T^*}{2}\right)(f)=0 \text{ on } D(T)\cap D(T^*)=D(T) \text{ as } D(T^*)=\mathcal{H}\]

which lead to \textit{formally} write $|T|=|\re T|$.

\end{document}